\documentclass[amscd,amssymb,verbatim,12pt]{amsart}

\usepackage[all]{xy}

\usepackage{pdfsync}
\usepackage{xcolor}
\usepackage{enumerate}
\usepackage{graphicx}
\usepackage{subfigure}
\newif\ifAMS
\IfFileExists{amssymb.sty}
  {\AMStrue\usepackage{amssymb}}
  {\usepackage{latexsym}}
  \usepackage{enumerate}
\theoremstyle{plain}
\newtheorem{Main}{Main Theorem}

\newtheorem{Thm}{Theorem}[section]
\newtheorem*{Main-nest}{Theorem \ref{T:nesting}}

\newtheorem{Cor}[Thm]{Corollary}
\newtheorem{Lem}[Thm]{Lemma}

\theoremstyle{definition}
\newtheorem{Def}[Thm]{Definition}

\theoremstyle{remark}

\newtheorem{Rem}{Remark}

\newtheorem{Ex}[Thm]{Example}
\newtheorem{No}{Notation}

\newcommand{\interior}{^{ \kern-5pt ^\circ}}
\newcommand {\bd}{\partial}

\newcommand {\iy}{\infty}

\newcommand {\N}{{\mathbb N}}
\newcommand {\R}{{\mathbb R}}
\newcommand {\Z}{{\mathbb Z}}

\newcommand{\Int}{\text{Int}\,}
\newcommand{\Ext}{\text{Ext}\,}

\newcommand {\cA}{ {\mathcal  A}}
\newcommand {\cB}{ {\mathcal  B}}

\newcommand {\cP}{{\mathcal  P}}
\newcommand {\cR}{{\mathcal  R}}
\newcommand {\cS}{{\mathcal  S}}

\newcounter{notes}
\newenvironment{Notes}
{\begin{list}{\textsc{Case} \Roman{notes}:}
		{\setlength\labelsep{10pt}
		\setlength \itemindent{10pt}
		\setlength \leftmargin{0pt}
		\setlength \labelwidth{0pt}
		\setlength \itemsep{3pt}
		\setlength \listparindent{12pt}
		\setlength \topsep{3pt}
		\usecounter{notes}}}
{\end{list}}	
\newcounter{subcase}
\newenvironment{Subcase}
{\begin{list}{\textsc{Subcase} \roman{subcase}:}
		{\setlength\labelsep{10pt}
		\setlength \itemindent{10pt}
		\setlength \leftmargin{0pt}
		\setlength \labelwidth{0pt}
		\setlength \itemsep{3pt}
		\setlength \topsep{3pt}
		\setlength \listparindent{12pt}
		\usecounter{subcase}}}
{\end{list}}			
\begin{document}

\title{From cuts to 
$\R$ trees}

\author
{Eric Swenson }

\subjclass{54F15 ,54F05,}

\address
[Eric Swenson] {Mathematics Department, Brigham Young University,
Provo UT 84602}
\email [Eric Swenson]{eric@math.byu.edu}

%

\begin{abstract} 
We provide sharp conditions under which a collection of separators $\cA$ of a connected topological space $Z$  leads to a canonical $\R$-tree $T$.  Any group acting on $Z$ by homeomorphisms will act  by homeomorphisms on $T$. 
\end{abstract}
\thanks{}
\maketitle

\section{Introduction}
The connection between separation and pretrees/trees has been studied extensively.  Whyburn  (\cite{WHY}) showed that  the cut points in  Peano continuum induce the structure of a dendrite. Ward (\cite{W}) saw that these cut points gave an axiomatic structure that we today refer to as a pretree.  Bowditch (\cite{BOW},\cite{BOW1} ,\cite{BOW5}) showed how to use this cut point pretree to go from an action on a continuum to an action on an $\R$-tree provide the action was a convergence action, and used this action on a tree to prove significant results about the boundaries of hyperbolic groups.

  The  Papasoglu and the author have shown how to go from  a metric continuum $Z$ with finite cuts to a canonical  ``cactus" $\R$ tree $T$ (
  \cite{SWE},\cite{SWE1}, \cite{PS}, \cite{PS2}).  Any action by homeomorphisms on $Z$ induces an action by homeomorphisms on $T$.  Using this we showed that the only minimal finite cuts that occur in CAT(0) boundaries are, as expected, cut pairs arising from the group in question virtually splitting over a virtually cyclic group   (\cite{PS1}, \cite{PS3}).
  
  A cut of a connected topological space $Z$ will be a closed set $A$ such that $Z\setminus A$ is not connected.  
  In this paper, the author will show how to go from a connected topological space $Z$ together with a collection of cuts (satisfying certain sharp axioms) to a pretree $\cP$.  Provide the set of cuts is in some sense ``locally countable", $\cP$ will canonically embed in an $\R$ tree.  The main difficultly is that we are no longer restricting the cuts to be finite subsets.  This complicates things significantly.
  
    A cut $A$ of $Z$ is minimal if for any $B \subsetneq A$, $Z \setminus B$ is connected.  One can always find minimal finite cuts, but in the setting of this paper minimal cuts simply will not exist.  Consider the squares $S_n$ of side length $\frac 1n$ for each $n \in \N$.  We glue an edge of $S_n$ to the unit interval $I=[0,1]$ in the obvious way, one edge of $S_n$ glued isometrically to the subinterval $[0,\frac 1 n]$.   This gives a Peano continuum $Z$ which embeds in $\R^3$.  Ever initial subinterval of $I$ is a cut of $Z$, but the point $0\in I$ is not a cut, so these cuts don't contain a minimal cut.
    
      This removes the main tool that we used in our previous work, namely that if $A$ is a minimal cut of a continuum $Z$ then there are subcontinuum $X$ and $Y$ with $Z = X \cup Y$ and $A = X \cap Y = \bd X= \bd Y$.
  
  As originally envisioned by the author, the setting of this theorem would be something along the lines of metric continuum, and the cuts would be closed nowhere dense sets.  He had grandiose visions of using $G_\delta$s and $F_\sigma$s and the Baire category theorem.  As one can see from the main theorem, none of this came to pass.  The author attempted to use these sorts of arguments and discovered that he couldn't find a way to make them work.  The only proofs that worked were proofs using only the ``technology" of the 3rd chapter of an undergraduate topology text.  Really.  The author is painfully aware that there is much more to topology than the idea of connected and of a separation, but nothing beyond that will be used in the paper you are currently reading.  Rational people would begin to question many of their life choices at this point.  Take this as the author's apology for that fact that although this result is vastly stronger than he could even have imagined (or even believed) when he started, the techniques employed are absurdly elementary to the point of being insulting to the reader.  

 \begin{Main} Let $Z$ be a connected topological space and $\cA$ a collection of closed subsets of $Z$ satisfying the following:
\begin{enumerate}
\item $\forall A \in \cA$,  $A$ separates $Z$.  
\item $\forall A,B \in \cA$, $A$ doesn't separate $B$.
\item For all distinct $A, B \in \cA$, $A \cap B$ doesn't separate $Z$.  
\end{enumerate}
Then there is a complete median pretree $\cP$ consisting of closed subsets $Z$ satisfying the following
\begin{enumerate}[(I)]
\item $\cA \subseteq \cP$ 
\item $Z = \cup \cP$ 
\item For $A, C \in \cP$ then $C$ doesn't separate $A$.
\item For $A,B,C \in \cP$ and  $A, B \not \subseteq C$, then 
\begin{enumerate}
\item if $C \in (A,B)$ then $C$ separates $A$ from $B$
\item if $C$ separates $A$ from $B$ then $D \in (A,B) $ for some $D \subseteq C$
\end{enumerate} 

\end{enumerate}
If $\cA \cap L$ is countable for each linearly ordered subset $L$ of  $\cP$, then $\cP$  is preseparable.
\end{Main}
It is known that a complete median pretree canonically embeds in an $\R$ tree \cite[Theorem 13]{PS1}. 

The author acknowledges that things like closed nowhere dense and the Baire category theorem may in fact be useful in establishing that $\cA$ satisfies this ``locally countable" condition and so that $\cP$ is preseparable, but is leaving this for later work and hopefully applications.  

We can use this result in the setting of the Cactus pretree \cite{PS2}.  We would take $\cA$ to be the set of all wheels and min cuts not contained in wheels.   The resulting pretree would be a subpretree of the Cactus pretree and the resulting $\R$ tree would be a subtree of the Cactus $\R$ tree (the Cactus $\R$ having some extra leaves).  

Papasoglu and the author intend to use this result to prove structure theorems about boundaries of groups and splitting theorems about nonpostively curved groups over higher rank  virtually Abelian subgroups.   

The author would like to acknowledge helpful conversation with Papasoglu, who in the end decided he had better things to do with his life than this.  

\section{ Separators to pretrees}
\subsection{Review of undergraduate topology}
\begin{Def}Let $Z$ be a topological space.  A separation of $Z$ is a pair $(U,V)$ of disjoint non-empty open sets of $Z$ such that $Z = U \cup V$.  The space $Z$ is connected if and only if it has no separation.   We define a  relation on $Z$ by two elements $x,y$ are equivalent if there is no separation $(U,V)$ of $Z$ with $ x\in U$ and $y\in V$.  This is easily seen to be an equivalence relation, and the equivalence classes are call the quasicomponents of $Z$.  It follows that any connected subset of $Z$ will be contained in a quasicomponent, also the quasicomponents are closed since their complement is a union of open sets, namely the other sets of the separations. 
\end{Def}
For $A,B\subseteq Z$, we say $A$ separates $B$ if we have $b,c \in B\setminus A$ with $b$ and $c$ in different quasicomponents of $Z\setminus A$.  For $A,B,C \subseteq Z$, we say $A$ separates $B$ from $C$ if there exists $b \in B\setminus A$ and $c \in C\setminus A$ with $b$ and $c$  in different quasicomponents of of $Z\setminus A$.  
\begin{No}  The symmetric difference of two sets $A, B$ will be denoted $A\triangle B= ( A\setminus B)\cup( B\setminus A)$, and the disjoint union of $A$ and $B$ will be denoted $A\sqcup B$ (and we will use this only when $A\cap B =\emptyset$).
In a topological space we will use $\Int$ for {\em interior} of a set,  $\Ext$ for the {\em exterior} of a set, and  $\bd $ for the boundary of  a set.  We remind the reader that for any topological space $Z$ and $A \subseteq Z$, $Z$ is the disjoint union of $\bd A$, $\Int A$, and $\Ext A$.  The $\Int A$ and $\Ext A$ are open sets and so $\bd A$ is a closed set. We define the closure of $A$ to be $\bar A = A \cup \bd A = (\Int A) \sqcup \bd A$ which is closed since its complement is the open set $\Ext A$.
\end{No}
\begin{Lem}\label{L:sep} Let $Z$ be a topological space and $E$ a closed set of $Z$.  If $F \subseteq Z$ with 
\begin{itemize}
\item $\bd F \subseteq E$ 
\item $F \not \subseteq E$ 
\item  $(Z \setminus F) \not \subseteq E$
\end{itemize}
then $E$ separates $Z$.
\end{Lem}
\begin{proof}  Consider the open sets $U = (\Int F )\setminus E$ and $V = (\Ext F )\setminus E$.   Since $\bd F \subseteq E$ but $F \not \subseteq E$  then $U \neq \emptyset$.  Similarly since $\bd (Z\setminus F) = \bd F \subseteq E$ but $(Z \setminus F) \not \subseteq E$, then $V= (\Ext F) \setminus E =[ \Int (Z\setminus F)] \setminus E \neq \emptyset$.  

Notice that if $z \not \in E$ then $z \not \in \bd F$ so $z \in \Int F \sqcup \Ext F$.  It follows that $Z \setminus E= U \cup V$.  Since $\Int F \cap \Ext F = \emptyset$ then $U \cap V = \emptyset$, and so 
 $(U,V)$ form a separation of $Z \setminus E$.
\end{proof}
 
 \begin{Cor} Let $A$ be a closed subset of the topological space $Z$  and  $(U,V)$ a separation of  $Z \setminus A$.  Then the boundary of $U$ in $Z$,  $\bd U \subseteq A$.   If $C \subseteq A$ is closed and  $C$ doesn't separate $Z$ then $\bd U \not \subseteq C$.  
 \end{Cor}
 \begin{proof} Since $A$ is closed, $U$ and $V$ are open sets of $Z$.  Thus $U = \Int U$ and $V \subseteq \Ext U$. It follows that $\bd U \subseteq A$.  
 
 Suppose that $\bd U \subseteq C$.  Since $U $ is nonempty then $U \not \subseteq C$.  Since $V$ is nonempty, $Z \setminus U \not \subseteq C$.  
  It follows from Lemma \ref{L:sep} that $C$ separates $Z$, a contradiction.
 \end{proof}
\begin{Lem}\label{L:qinc} Let $Y\subseteq Z $ where $Z$ is a topological space.  The inclusion map $\iota: Y \hookrightarrow Z$ sends quasicomponents to quasicomponents.
\end{Lem}
\begin{proof}$\iota$ is a continuous function so for any separation $(U,V)$ of $Z$,  $(\iota^{-1}(U), \iota^{-1}(V))$ is either a separation of $Y$, or one of 
$\iota^{-1}(U), \iota^{-1}(V)$ is empty.  The result follows.  
\end{proof}
\subsection{Cuts and Blobs}
Let $Z$ be a connected topological space  and $\cA$ a collection of closed subsets of $Z$ (called {\em cuts}) satisfying the following conditions:
\begin{enumerate}
\item For each $A \in \cA$, $A$ separates $Z$.
\item For any $A,B \in \cA$, $A$ doesn't separate $B$.
\item  For any $A,B \in \cA$ distinct, $A \cap B$ doesn't separate $Z$.
\end{enumerate}

We will show there is a canonical pretree encoding the separation properties of $\cA$.  
Conditions (2) and (3) are both individually necessary.  For condition (2) this is shown in Example \ref{Ex: corners}.  The necessity of condition (3) in 3-manifolds is literally the reason cube complexes were invented \cite{SAG}. 

 \begin{Def} A subset $P$ of $Z$ is called {\em inseparable} is there is no element of $\cA$ which separates $P$.  Clearly a nested union of inseparable set is inseparable, so by Zorin's Lemma every inseparable set is contained in a  maximal inseparable sets.  Notice that a maximal inseparable set will be closed because its complement is open (That is for any maximal inseparable $B$ of $Z$ and any $x \not \in B$ there is a cut $A \in \cA$ and a separation $(U,V)$ of $Z\setminus A$ with $x \in U$ and $U \cap B = \emptyset$.  Thus $Z \setminus B$ is open.)
 
Since cuts are inseparable,   every cut is contained in a maximal inseparable set and some cuts { \bf could be maximal inseparable sets}.  We define $\cB$ to be the maximal inseparable sets which are not   just a single cut.  We will call the elements of $\cB$ {\em blobs} and we define our pretree-to-be $\cP = \cA \sqcup \cB$. In particular $\cP$ contains every maximal inseparable set.
 \end{Def}
Every quasi component of $ Z \setminus(\cup \cA)$ is contained in a blob, but the blobs can contain more than one quasi component of $ Z \setminus(\cup \cA)$, to wit:
 \begin{Ex} Let $Z$ be the union of $\R^2$ with a half plane glued along the horizontal axis.  For $i \in \Z$, let $A_i$ be the union in $\R^2$ of the vertical ray $\{i\} \times[0,\iy)$ and  the horizontal ray $[i,\iy) \times \{0\}$ (both rays emanate from the point $(i,0)$). $\cA = \{A_i:i \in \Z\}$ satisfies our four conditions.  The open half plane  $ \R  \times(-\iy,0)$  and the glued on open half plane are different quasi components of $ Z \setminus(\cup \cA)$ but their union is contained in a single blob.    \end{Ex}
 \includegraphics[keepaspectratio, scale =1]{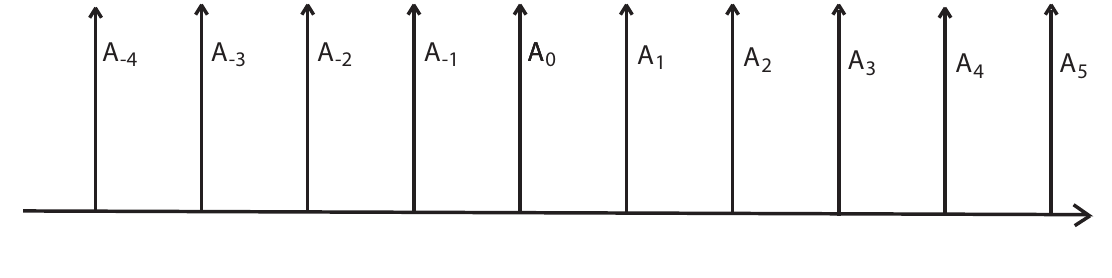}

 \begin{Rem}{\bf Instead} of the above definition of blob, we could do the following:  For $x, y \in Z \setminus(\cup \cA)$ we say $x \sim y$ if there is no $C \in \cA$ separating $x$ from $y$.  This is easily seen to be an equivalence relation on $ Z \setminus(\cup \cA)$.  Let $\cB$ be the set of equivalence class.  The proof of the pretree axioms in this case would follow mutatis mutandis from the treatment below. The pretree so obtained would be a subpretree of the one above; it would be missing exactly those blobs which are unions of cuts.
 \end{Rem}
 We define the betweenness relation on $\cP$ in stages. 
 \begin{Def} \label{D:int} For $B,C \in \cP$ and $A \in \cA$, we say that $A \in (B,C)$ if  $A$ separates $B$ from $C$.  
 For $B,C \in \cP$,  we define $[B,C]= (B,C)\cup \{B,C\}$ and similarly for half open intervals.
 \end{Def}
 We will need the following easy result.
 \begin{Lem}\label{L:quas} For distinct $A,B,C \in \cP$ with $A \in \cA$ and $A \in (B,C)$  then $B\setminus A$ is nonempty and in a quasicomponent of $Z\setminus A$, and $C\setminus A$ is also nonempty and in a different quasicomponent of $Z\setminus A$.  
 \end{Lem}
 \begin{proof} Nonempty follows from condition (3) on $\cA$ or the definition of $\cB$.
The other follows from definition of $(B,C)$ and from either  condition (2) on $\cA$ or the definition of $\cB$. \end{proof}
\begin{Lem}\label{L:blob} For $B,C $ distinct maximal inseparable sets, there is $A \in \cA$ with $A \in (B,C)$.  Furthermore $B\cap C \subseteq A$. 
\end{Lem}
\begin{proof} Since $B$ and $C$ are maximal inseparable, $B \cup C$ is separable.  Thus there is $A\in \cA$ separating two points of $B\cup C$, but since $B$ and $C$ are inseparable, $A \in (B,C)$.  
There is a separation $(U,V)$ of $Z\setminus A$ with $B \setminus A \subset U$ and $C \setminus A\subset V$.  This implies $\left(B \setminus A\right) \cap \left(C \setminus A\right) = \emptyset$ which implies $B \cap C \subseteq A$.
\end{proof}
 We now extend the betweenness relation.
 \begin{Def}  For distinct  $A, B,C \in \cP$ with  $B\in \cB$ we say that $B \in (A,C)$ if $[A,B)\cap (B,C] = \emptyset$, where the $[B,A)$ and $(A,C]$ are the  intervals defined in Definition \ref{D:int} . We extend to closed (and half open intervals) as before.
 \end{Def}  
 We must now show that $\cP$ is a pretree. The axioms of a pretree are the following:
 \begin{enumerate}
 \item $\forall B \in \cP, \,(B,B) = \emptyset$
 \item $\forall A,B \in \cP,\, (A,B) = (B,A)$
 \item $\forall\, B,C \in \cP$, if $A \in (B,C)$ then $B \not \in (A,C)$
 \item $\forall \, A,B,C \in \cP$, $(A,C) \subseteq (A,B] \cup [B,C)$
 \end{enumerate}
 The first two axioms are  satisfied by our definition of betweenness on $\cP$.  Axiom (3) will be verified  in Theorem \ref{T:axiom 3}  and Axiom (4) in Theorem \ref{T:axiom 4}.
 \begin{Ex}\label{Ex: corners}  Condition (3) on $\cA$ is new and we now give an example to show that condition (3)  is necessary.  
 
  \includegraphics[keepaspectratio, scale =1.5]{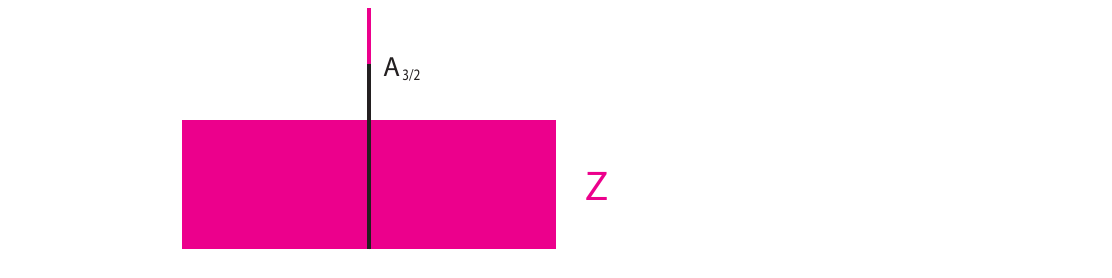}
 
 Consider the subset of the plane $Z = \left([-1,1]\times [0,1] \right)\cup\left(\{0\} \times[1,2]\right)$.  For $y \in [1,2]$ let $A_y = \{0\}\times [0,y]$ and set $$\cA = \{ A_y :\, y \in [1,2]\}.$$  Clearly $Z$ and $\cA$ satisfy conditions (1) and (2) but not condition (3).  There are exactly two blobs, $B= [-1,0] \times [0,1]$ and $C = [0,1] \times [0,1]$ and  $(B,C)=\cA  $.  However $(B, A_1] \cap [A_1, C) = \{A_1\}$ so 
 $(B,C) \not \subseteq  (B, A_1] \cap [A_1 \cap C)$ and thus pretree axiom (4) fails in this example.
 \end{Ex}
 \begin{Thm} \label{T:axiom 3} $\forall B,C \in \cP$, if $A \in (B,C)$ then $B \not \in (A,C)$.
 \end{Thm}
 \begin{proof} Assume by way of contradiction that we have $A,B,C \in \cP$ with  $A \in (B,C)$ and $B \in (A,C)$.
 
\begin{Notes} \item \label{3,1} $A, B \in \cA$.  \hfill\break
Since $A \in (B,C)$ then there is a separation $(U_1,V_1)$ of $Z\setminus A$ with $B\setminus A \subseteq U_1$ and $C\setminus A \subseteq V_1$. Similarly there is a separation $(U_2,V_2)$ of $Z\setminus B$ with $A\setminus B\subseteq U_2$ and $C\setminus B \subseteq V_2$.
Notice that $U_1,U_2,V_1,V_2$ are all open sets of $Z$.   Let $V =  V_1\cap  V_2$.  

We will show that $A\cap B$ separates $Z$.   We first show that $V $ is nonempty. By condition (3) on $\cA$, $C \setminus A$ and $ B\setminus A$ are nonempty.  Since $  B \setminus A \subseteq U_1$ and $C \setminus A \subseteq V_1$, it follows that $C\setminus A \not \subseteq B \setminus A$ which implies that $C \not \subseteq A \cup B$.
 Thus  $$V \supseteq (C \setminus A) \cap (C \setminus B) = C\setminus(A\cup B) \neq \emptyset,$$ so $V$ is anonempty open set of $Z$.

 Since $U_1 \cup U_2 \subseteq Z \setminus V$, then $Z\setminus V \not \subseteq A \cap B$.  
Notice that  $$\bd V =\bd(V_1 \cap V_2) \subseteq (\bd V_1) \cup(\bd V_2) \subseteq A \cup B$$
However, the $A \triangle B \subseteq (U_1 \cup U_2)$  which is an open set disjoint from $V$ and so contains no boundary points of $V$.  Thus $\bd V \subseteq A \cap B$.  Clearly $V \cap (A \cap B) = \emptyset$, so by Lemma \ref{L:sep}, $A\cap B$ separates $Z$, which contradicts (3).

  \item $\{A,B\} \not \subseteq \cA$ \hfill\break
 Say $ A  \in \cB$.  Since $A \in (B,C)$ by definition  $[B,A) \cap (A, C]\cap \cA = \emptyset$.  Thus if $B \in \cA$ then $B \not \in (A,C] \supseteq (A,C)$, a contradiction. 
 
 We are left with the sub-case where $B  \in \cB$.  Since $B \in (A,C)$ then $[A,B) \cap (B,C]\cap \cA = \emptyset$. Since $A \neq B$ there is $D \in \cA$ with $D \in (A,B)$.  Thus there is a separation $(U, V)$ of $Z \setminus D$ with $A \subseteq U$ and $B \subseteq V$.  Either $C \setminus D \subseteq U$ or $C \setminus D \subseteq V$.  If $C \setminus D \subseteq U$, then $D \in (A,B) \cap(B,C)\cap \cA= \emptyset$ a contradiction.  Similarly if $C \setminus D \subseteq V$ then $ D \in (B,A) \cap (A,C) \cap \cA= \emptyset$ again a contradiction.
\end{Notes}
 \end{proof}
 \begin{Lem}\label{L:refine}If $(U_1,V_1)$ and $(U_2,V_2)$ are separations of the space $Z$, and if $U_1\cap U_2 \neq \emptyset$, then $(U_1 \cap U_2, V_1\cup V_2)$ is a separation of $Z$.
 \end{Lem}
 \begin{proof}  Clearly $U_1\cap U_2$ and $V_1\cup V_2$ are both nonempty open sets.  $$ [U_1 \cap U_2]\cap[V_1 \cup V_2]=[(U_1\cap V_1) \cap U_2]\cup [U_1 \cap(U_2 \cap V_2)  ] = \emptyset \cup \emptyset = \emptyset$$
$$ [U_1 \cap U_2]\cup[V_1 \cup V_2]=[(U_1\cup V_1)\cup V_2]\cap [V_1 \cup (U_2\cup V_2)]= Z \cap Z = Z$$
 \end{proof}
  \begin{Lem}\label{L:contain} Let  $A,C \in \cP$ and $B,D \in \cA$ all distinct.  Suppose that $(U,V)$ is a separation of $Z \setminus B$ with $A \setminus B \subseteq U$ and $C \setminus B \subseteq V$, and that $(O, W)$ is a separation  of $Z \setminus D$ with $B \setminus D \subseteq O$ and $C \setminus D \subseteq W$, then  $U \cap W = \emptyset$, and in fact $U \subsetneq O$.
 \end{Lem}
 \begin{proof} By definition, $B \in (A,C)$ and $D \in (B,C)$.  By theorem \ref{T:axiom 3} $B \not \in (D,C)$ so $D \setminus B \subset V$.  Suppose that $U \cap W \neq \emptyset$.    We will show that $B \cap D$ separates which contradicts condition (3) on $\cA$.
 Clearly $\bd (U \cap W) \subseteq (\bd U) \cap (\bd W) \subseteq B \cup D$.  Notice that  $B\setminus D\subseteq  O$ and so $(B \setminus D)\cap \bd (U\cap W)= \emptyset$ and similarly $D\setminus B\subseteq V$ and so $D \setminus B \cap \bd (U\cap V)= \emptyset$.  It follows that $\bd (U \cap W) \subseteq B \cap D$.  Since $  O \cup V \subseteq Z\setminus(U \cap W)$ and $(O \cup V )\cap (B \cap D) =\emptyset$ by Lemma \ref{L:sep} $B \cap D$ separates $Z$ which is a contradiction.  Thus $U \cap W = \emptyset$. 
 
  It follows that $U \subseteq O \cup D$.  Since $D \in (B,C), $ by Theorem \ref{T:axiom 3} $B \not \in (D,C)$ and so $D \setminus B \subseteq V$.  Thus $D \subseteq [B\cup V]$ and  $U \cap D \subseteq U \cap [B \cup V]= \emptyset$ .  We have shown  $U \subseteq O$.  Since $B \cap U = \emptyset $ and $\emptyset \neq B\setminus D\subseteq O$, then $U \subsetneq O$.

  \end{proof}
 \begin{Cor}\label{C:contain} Let  $A,C \in \cP$ and $B \in (A, C) \cap \cA$.  Then $(B,C) \cap \cA \subseteq (A,C)$.  
 \end{Cor}
 \begin{proof} Let $D \in (B,C) \cap \cA$.  Since $B \in (A,C)\cap \cA$ there is a separation $(U,V)$ of $Z \setminus B$ with $A \setminus B \subseteq U$ and $C \setminus B \subseteq V$.   
  Since $D \in (B,C) \cap \cA$, there is a separation $(O, W)$ of $Z \setminus D$ with $B \setminus D \subseteq O$ and $C \setminus D \subseteq W$.  By Lemma \ref{L:contain} $U \cap W = \emptyset$. 
  
   It suffices to show that $A \setminus D \subseteq O$.  {\bf Suppose not}, then $A\setminus D \subseteq W$.  Since $A \setminus B \subseteq U$ and $D \setminus B \subset V$, then $(A\setminus B) \cap (D\setminus B)=\emptyset$.  Similarly since $A \setminus D \subseteq W$ and $B \setminus D \subseteq O$, then  $(A\setminus D) \cap (B\setminus D)=\emptyset$.
   It follows that $$A \cap B =A \cap D \subseteq B \cap D\,.$$ 
   We now have $\emptyset =U \cap W \supseteq A\setminus(B \cup D) = A\setminus B =A\setminus D \neq \emptyset$ by condition (3) on $\cA$,  a contradiction.
   
 \end{proof}
 \begin{Thm}\label{T:axiom 4}  $\forall \, A,B,C \in \cP$, $(A,C) \subseteq (A,B] \cup [B,C)$
 \end{Thm}
 \begin{proof} If $A,B,C$ are not distinct, the result is obvious so we assume $A,B,C$ distinct.  Let $D \in (A,C)$. We may assume that $D \neq B$.  We must show $D \in (A,B) \cup (B,C)$.
 \begin{Notes}
 \item $D \in \cA$ \hfill\break
 There is a separation $(U,V)$ of $Z\setminus D$ with $A\setminus D \subseteq U$ and $C \setminus D \subseteq V$.  Since $D$ doesn't separate $B$, either $B\setminus D \subseteq U$ or $ B \setminus D \subseteq V$. If  $B\setminus D \subseteq U$ then by definition $D \in (B,C)$ and if $ B \setminus D \subseteq V$ then $D \in (A,B)$.
 \item $D \in \cB$\hfill\break
 Then by definition $[A, D) \cap (D, C]\cap \cA = \emptyset$. Suppose $D \not \in (A, B)\cup (B,C)$, then by definition there is $E \in [A,D) \cap (D,B] \cap \cA$ and there is $F \in [B,D) \cap (D,C] \cap \cA$.  Since $[A, D) \cap (D, C]\cap \cA = \emptyset$, $E \neq F$.

 \begin{Subcase}
 \item $E \in [A,D) \cap (D,B) \cap \cA$ and $F \in (B,D) \cap (D,C] \cap \cA$.\hfill\break
 Let $(O,W)$ be a separation of $Z\setminus E$ with $B\setminus E \subseteq O$ and $D  \subseteq W$.  Either $F\setminus E \subseteq O$, or $F\setminus E \subseteq W$.  

Consider the case  $F\setminus E \subseteq O$, so  $E \in (F,D)$.  Since $F \in [C,D) \cap \cA$,  by Corollary \ref{C:contain}  $E \in(F,D) \cap \cA \subseteq (C, D)$ and we have the contradiction  $$E \in [A, D) \cap (D, C]\cap \cA = \emptyset.$$

Now consider the  case  $F\setminus E \subseteq W$, so $E \in (B,F)$.  By Theorem \ref{T:axiom 3} $F \not \in (B,E)$ so by Case I $F \in (E,D)$.
Since $E \in [A,D) \cap \cA$ then by Corollary \ref{C:contain}  $F \in (A,D)$ and we have the contradiction  $$F \in [A, D) \cap (D, C]\cap \cA = \emptyset.$$
\item $E =B$ or $F=B$  \hfill\break
Say $B=E \in [A,D)$.   Since $F \in [B,D)$  by Corollary \ref{C:contain} $F \in [A,D)$ again with the contradiction 
$$F \in [A, D) \cap (D, C]\cap \cA = \emptyset.$$
 
 \end{Subcase}
 \end{Notes}
 \end{proof}
 Thus we have shown that $\cP$ is a pretree.  
  \begin{Def}
 A pretree $\cR$ is called preseparable if for any linearly ordered subset $L$ of $\cR$, there is a countable subset $Q \subseteq L$ such that for any distinct $a,b \in L$, $[a,b] \cap Q \neq \emptyset$.  
 \end{Def}
 \begin{Cor}\label{C: preseparable}  If for any linearly ordered subset $L$ of $\cP$, $L \cap \cA$ is countable  then $\cP$ is preseparable.  
 \end{Cor}
 \begin{proof} Let $L$ be a linearly ordered subset of $\cP$ and let $A,B \in L$ distinct.  We may assume $A, B \in  \cB$.  By
 Lemma \ref{L:blob}  $(A,B)\cap \cA \neq \emptyset$, so $\cA\cap L$ is the countable set of the definition. 
 \end{proof} 
 We recall some definitions and properties of pretrees due to 
 Bowditch \cite{BOW5}
 \begin{Def}In a pretree $\cR$, every interval has a  linear order (exactly two actually) so that the subintervals are the same in the pretree as in the linear order.  Such a linear order will be called a {\em natural linear order}  A subset $X \subset \cR$ is called linearly ordered, if $X$ has a linear order so that the intervals of the linear oder on $X$ are the same intervals in $\cR$.  We will also refer to such a linear order as a {\em natural  linear order}.   A subset $Y$ of a pretree $\cR$ is called convex if for any $A,B \in Y$, $[A,B] \subset Y$.   
 \end{Def} Clearly a convex subset of a linearly ordered space will be linearly ordered (by the restriction of the natural linear order).  Also the intersection of convex sets is convex.
 \begin{Thm}\label{T:BOW} \cite{BOW5}  If $\cR$ is a pretree and $x,y,z,w \in \cR$ then
 \begin{enumerate}
 \item If $y \in (x,z)$ and $z \in (y,w)$ then $x<y<z<w$ for a natural linear order on $[x,w]$
  \item If $y \in (x,w)$ and $z \in (y,w)$ then $x<y<z<w$ for a natural linear order on $[x,w]$
  \item If $y \in (x,w)$ and $z \in (x,w)$ then there is a natural linear order on $[x,w]$ with either 
  $x<y<z<w$ or $x<z<y<w$.
 \end{enumerate}
 \end{Thm}
 \begin{Def} For $\cR$ a pretree,
we say $A,B \in \cR$ distinct  are adjacent if $(A,B) = \emptyset$.
 \end{Def}
 \begin{Lem}\label{L:cutblob}
 If two elements of $\cP$ are adjacent, then one of them is a cut and the other is a blob and said cut is contained in said blob.
 \end{Lem}
 \begin{proof}By Lemma \ref{L:blob} they cannot both be blobs. Say then that $A,B \in \cP$ adjacent with $A \in \cA$.

Suppose $B \in \cA$.  Then since $(A,B) = \emptyset$, it follows that $A \cup B$ is inseparable, and so $A\cup B$ is contained in some blob $C$.  Since $[A,C)\cap (C,B] = \emptyset$, by definition $C \in (A,B)$ a contradiction.

Thus $B \in \cB$, and since $(A,B) = \emptyset$, $A \cup B$ is inseparable which implies $A \cup B = B$ as required.  
 \end{proof}

 \begin{Def}
 We say that a pretree $\cR$ is complete if every interval is complete as a linearly ordered topological space.  That is every nonempty subset with an upper bound has a supremum.  
 \end{Def} 
  \begin{Def}  For $A,B,C \in \cR$ a pretree, we say  the {\em median} of $A,B,C$ is the intersection $[A,B]\cap [B,C]\cap [A,C]$. 
 It is known that the median is a most one point \cite{BOW5}, but it can be empty.  A pretree $\cR$ is called a {\em median} pretree if no medians of $\cR$ are empty. 

 \end{Def} 
 \begin{Thm}\label{T: Median} If $ \cP$ is complete, then $\cP$ is a medium pretree.
 \end{Thm}
 \begin{proof}
 Let $A,B,C \in \cP$.  Suppose by way of contradiction that $[A,B]\cap [B,C]\cap [A,C]= \emptyset$.   Clearly we may assume that $A,B,C$ are distinct.  Put the natural linear order on $[A,B]$ with $A<B$.  
 \begin{Notes} 
 \item There exist $ \hat A, \hat B, \hat C \in \cP$ with $[A,B] \cap [A,C]= [A,\hat A]$, $[B,A] \cap [B,C] = [B,\hat B]$ and $[C,A]\cap [C,B] = [C,\hat C]$.\hfill\break
 Notice that if $\hat A \in [ \hat B, B]$ then $\hat A \in ([A,B] \cap [A,C])\cap ([B,A] \cap [B,C] )=[A,B]\cap [B,C]\cap [A,C]$ contradicting our supposition, so $\hat A \not \in [\hat B,B]$.  Thus by Theorem \ref{T:BOW}(3) $A<\hat A<\hat B<B$.

By axiom (4) $[A,B] \subseteq [A,C] \cup [C,B]$ so $$[A,B] = ([A,C] \cap [A,B]) \cup([C,B]\cap [A,B]) =[A,\hat A] \cup [\hat B ,B]$$ Since $A<\hat A<\hat B <B$, it follows that $(\hat A, \hat B)=\emptyset$, and by symmetry, $\hat A,\hat B$ and $\hat C$ are pairwise adjacent. By Lemma \ref{L:cutblob}  each two element subset of $\{\hat A, \hat B,\hat C\}$ contains one blob and one cut which is of course impossible.

  \item Relabeling if need be, for $D$ the supremum of $[A, B] \cap [A,C]$ in the interval $[A,B]$, $D \not \in [A,C]$.  \hfill\break
  Notice that $[A,D)= [A, B] \cap [A,C]$.
  Let $E$ be the supremum of the set  $[A, B] \cap [A,C]$ in the interval $[A,C]$ with the natural linear order on $[A,C]$ yielding $A<C$. Notice that our two natural linear orders agree on $[A, B] \cap [A,C]$.
  If $E \in [A,B]$ then $D =E$ contradicting $D \not \in [A,C]$.  {\bf Thus} $E \not \in [A,B]$, and as above $[A,E) =[A, B] \cap [A,C]= [A,D)$
  
  Claim: $D$ and $E$ are adjacent. \hfill\break 
Suppose $F\in (D,E)$.  Then by axiom (4) $F \in [A,D) \cup [A, E) =[A,D) =[A,E)$.  Since $F$ is not an upper bound on $[A,D)$, then $F<G $ for some $G \in [A,D)$.
Thus $G \in (F,D)= (F,E)$, and  by axiom (3)  $F \not \in [G,D] \cup [G,E]\supset [D,E]$ contradicting $F \in (D,E)$, and the claim is proven. 

Thus $D$ and $E$ are adjacent, so one of them is a cut and the other a blob containing said cut.  Say $D$ is a cut and $E$ is a blob with $D \subseteq E$.    Notice that by axiom (4), $[B,E) \subseteq [B,D] \cup [D,E) = [B,D]$.  Since $[B,D] \cap [A,C] =\emptyset$ then $[B,D] \cap [A,E) = \emptyset$ so $[B,E) \cap (E, A] = \emptyset$ and by definition $E \in [A,B]$ a contradiction.

 \end{Notes}
 \end{proof}

 \begin{Thm}\label{T:complete} $\cP$ is complete.
 \end{Thm}
 \begin{proof} Let $A,B \in \cP$ distinct   and $\cS \subset [A,B]$  nonempty.  By replacing $\cS$ with its convex hull, we may assume $\cS$ is convex.  Take the natural linear order on $[A,B]$ with $A<B$.  By way of contradiction, suppose that $\cS$ has no supremum in this linear order, then for any $C \in [A,B]$ either $(C,B) \cap \cS$ is infinite, or there are infinitely many $D \in (A,C)\cap \cA$ with $\cS \subseteq [A,D)$.

For each $C \in (A,B) \cap \cA$ choose a separation  $(U_C,V_C)$  of $Z\setminus C$ with $A \setminus C \subseteq U_C$ and $B \setminus C\subseteq V_C$.  Let $$U = \bigcup\limits_{C \in \cS \cap \cA}U_C\cup C\,.$$  
 By Lemma \ref{L:contain} and Theorem \ref{T:axiom 3} this is a nested union, that is if $C, \hat C \in   \cA \cap (A,B)$ with $C<\hat C$ then $ U_C \cup C \subsetneq U_{\hat C} \cup \hat C$ (and   $ V_{\hat C}  \subsetneq  V_C$).  Notice that even though the choices of $(U_C,V_C)$ are not unique, for any choices we will have this nesting, and it follows that $U$ is independent of these choices.
 We let $$V= \bigcap\limits_{C \in \cS\cap \cA} V_C\,. $$    Notice that since $Z = U_C \sqcup V_C \sqcup C$ for all $C \in \cS\cap \cA$ then $Z = U \sqcup V$. It follows that $V$ is also independent of these choices.
 
 Notice that there is $D \in(A,B) \cap \cA$ so that $\cS \subset [A,D)$ which implies by condition (3) on $\cA$ that  that $U \subseteq U_D$ and $B\setminus D \subseteq V_D$.  By Lemma \ref{L:sep}, $\bd U$ separates $Z$ (in particular $\bd U \neq \emptyset$).  By condition (3) on $\cA$, $\bd U \not\subseteq D\cap E$ for any distinct $D,E \in \cA\cap(A,B)$.

   We now show that $\bd U$ is inseparable.  Suppose not, then there are $o,w \in\bd  U$, $E \in \cA$ and $(O,W)$ a separation of $Z \setminus E$ with $o\in O$ and $w \in W$.  By definition of boundary point, there is $ \hat o \in O \cap (U_{C_1}\cup C_1)$ for some $C_1 \in \cS\cap \cA$ and $o' \in O \cap V_D$ for all $ D \in \cS\cap \cA$.  Similarly there is $\hat w \in W \cap (U_{C_2}\cup C_2)$ for some $C_2 \in \cS\cap \cA$ and $w' \in W \cap V_D$ for all $D \in \cS\cap \cA$.  By nesting, for $C > \max\{C_1,C_2\}$, $ \hat o, \hat w \in U_C\cup C$, and of course $o',w' \in V_C$.  Since there are infinitely many such $C$, we may assume that $C \neq E$. By moving $A$ or $B$ if needed, we may assume that $A,B,C,E$ all distinct. 
   
   Suppose that $E \in (A,B)$.  By symmetry, we may assume that $A \setminus E \subseteq O= U_E$ and $B \setminus E \subset W = V_E$.   If $E$ is an upper bound on $\cS$ then $\hat w \in U_C \cup C \subseteq O\cup E$ contradicting $\hat w \in W$. If $E$ is not an upper bound on $\cS$ then we may assume that $E <C$ and so $o' \in  V_C \subseteq E \cup W$ contradicting $o' \in O$.  Thus $E \not \in (A,B)$ and by symmetry, we may assume that $A\setminus E, B \setminus E \subseteq O$.  It follows from Theorem \ref{T:BOW}  that  $C \setminus E \subseteq O$ as well. Since $Z = \cup \cA$ we may choose $F \in \cA$ with $F\setminus E \subseteq W$.
   
    By pretree axiom (4), $C \in ( A, E) \cup (E,B)$.
   \begin{Notes}
   \item $C \in (A,E)$ \hfill\break
   By Lemma \ref{L:refine} we may assume that $E \setminus C \subset V_C$.    Applying Lemma \ref{L:contain} to $A,C,E,F$ we see that $U_C \cap W = \emptyset$.  Since $C\setminus E \subseteq O$, we see that $[U_C \cup C] \cap W = \emptyset$, contradicting $\hat w \in [U_C \cup C] \cap W $.  
  
   \item $C \in (E,B)$ \hfill\break
   By Lemma \ref{L:refine} we may assume that $E \setminus C \subset U_C$.  Applying Lemma \ref{L:contain} to $F,E,C,B$ we see that $W \cap V_C = \emptyset$.  Contradicting $w' \in V_C \cap W$.  
   \end{Notes}
   Thus have we shown that $\bd U$ is inseparable.

   Since $\bd U$ is inseparable there is $F \in \cB$ with $\bd U \subseteq F$. 
   \begin{Notes}
   \item $F \in (A,B)$. \hfill\break
   We first show that $F$ is an upper bound on $\cS$.  Suppose not, then there is $C \in \cS \cap \cA$ such that $F\setminus C \subset U_C$ by Lemma \ref{L:refine} and nesting.  It follows that $\bd U \subseteq F \subseteq U_C \cup C$.  Clearly $U_C \subseteq \Int U$ so $\bd U \subseteq C$.  By nesting and the same argument, $\bd U \subseteq D$ for all $D \in cS \cap\cA$ with $D\ge C$.  Since $\bd U$ separates this violates condition (3) on $\cA$ and we have a contradiction.   Thus $F$ is an upperbound on $\cS$.   
   
   Thus there infinitely many  $D \in (A,F) \cap \cA$ such that $\cS \subseteq [A,D)$.  By nesting $U_C \cup C \subseteq U_D \cup D$ for all $C \in \cS \cap \cA$, so $U$ is contained in the closed set $ U_D \cup D$ implying that $\bd U \subseteq U_D \cup D$.  By Lemma \ref{L:refine} we may assume that $F \setminus D \subseteq V_D $, so $\bd U \subseteq D \cup V_D$.  Thus $\bd U \subseteq [U_D \cup D] \cap [D \cup V_D] = D$.  Thus $\bd U$ is contained in infinitely many such $D \in (A,F)\cap \cA$ contradicting condition (3) on $\cA$.
   \item $F \not \in (A,B)$.  \hfill\break
   Then there is $G \in (A,F)\cap (F,B) \cap \cA$.
  Using Lemma \ref{L:refine} there is a separation $(X,Y)$ of $Z \setminus G$ such that $A\setminus G, B\setminus G \subset X$ and $F \setminus G \subseteq Y$.  Since $\bd U = \bd V \subseteq F \subseteq Y\cup G$.  Then $U \cap X = [\Int U] \cap X$ is open and similarly $V \cap X=[ \Int V ]\cap X$ is open.  Since $U \sqcup V = Z$, it follows that $X = [U\cap X] \sqcup [V \cap X]$.  Thus $( [U\cap X]\cup Y, [V\cap X])$ is a separation of $Z\setminus G$.   Since $A\setminus G \subseteq U\cap X$ and $B \setminus G \subset V \cap X$  then $G \in (A,B)$, and we may assume that $U_G =  [U\cap X]\cup Y$ and $V_G =V \cap X$.  In particular $F \subseteq U_G \cup G$.  
  
 {\bf Claim}: $G$ is an upper bound on $\cS$. \hfill\break
  Suppose not, then there are infinitely  $C \in \cS \cap \cA$ such that $C> G$, so $U_G \cup G \subsetneq U_C \cup C$.  Since $U_C \subseteq \Int U$,  and $\bd U \subseteq F \subseteq U_G \cup G \subseteq U_C\cup C$ it follows that $\bd U \subseteq C$ for all such $C$.  As before this violates condition (3) on $\cA$ and we have a contradiction.  
  
  Thus $G$ is an upper bound on $\cS$, and there are infinitely many $D \in (A, G)\cap \cA$ with $\cS \subseteq [A, D)$.  We now go back and change our choice of $(U_G, V_G)$ so that $U_G  = U\cup X$ and $V_G = [V\cap X] \cup Y$.  Notice that now $\bd U \subset V_G \cup G$.  Since $U \subseteq U_D \cup D$ which is closed then $\bd U\subseteq U_D \cup D$.  Thus $\bd U \subseteq [U_D \cup D]\cap [D \cup V_D]= D$ for infinitely many $D \in (A, G)\cap \cA$.  Once again, this violates condition (3) on $\cA$ and we have our final contradiction.
  \end{Notes}
 \end{proof}
 We now complete the proof of the main theorem which we state again for ease of reading
 \begin{Main} Let $Z$ be a connected topological space and $\cA$ a collection of closed subsets of $Z$ satisfying the following:
\begin{enumerate}
\item $\forall A \in \cA$,  $A$ separates $Z$.  
\item $\forall A,B \in \cA$, $A$ doesn't separate $B$.
\item For all distinct $A, B \in \cA$, $A \cap B$ doesn't separate $Z$.  
\end{enumerate}
Then there is a complete median pretree $\cP$ consisting of closed subsets $Z$ satisfying the following
\begin{enumerate}[(I)]
\item $\cA \subseteq \cP$ 
\item $Z = \cup \cP$ 
\item For $A, C \in \cP$ then $C$ doesn't separate $A$.
\item For $A,B,C \in \cP$ and  $A, B \not \subseteq C$, then 
\begin{enumerate}
\item if $C \in (A,B)$ then $C$ separates $A$ from $B$
\item if $C$ separates $A$ from $B$ then $D \in (A,B) $ for some $D \subseteq C$
\end{enumerate} 

\end{enumerate}
If $\cA \cap L$ is countable for each linearly ordered subset $L$ of  $\cP$, then $\cP$  is preseparable.
\end{Main}
\begin{proof}
The proof that $\cP$ is a pretree is in Theorems \ref{T:axiom 3} and \ref{T:axiom 4}. Theorem \ref{T:complete} shows that $\cP$ is complete after which Theorem \ref{T: Median} shows that it is median.  By Corolary \ref{C: preseparable}, $\cP$ is preseparable when the ``locally countable" condition is satisfied. 

We are given that the elements of $\cA$ are closed sets and as we argued in the definition the blobs are closed since thier complements are open.  Thus $\cP$ is a collection of closed sets.  Points are inseparable, therefor every point is contained in a maximal inseparable set which will be an element of $\cP$, thus $\cup \cP = Z$.  By definition $\cA \subseteq \cP$.  Thus we have proven (I) and (II). 

We now prove (III), that is: For $A,C \in \cP$,  $C$ doesn't separate $A$.   Clearly we may assume that $A \not \subseteq C$.  For $C\in \cA$, (III) is just Lemma \ref{L:quas}.  

 Now consider the case where $C \in \cB$.  Let $a, d \in A \setminus C$ be distinct.  Since $a,d \not \in C$ there are cuts separating each of them from $C$.  Using Theorem \ref{T:BOW}(3), and Lemma \ref{L:contain} there is a single 
 $H \in (A,C)\cap \cA$ and a separation $(O,W)$ of $Z \setminus H$ with $a,d \in O$ and $C \setminus H \subseteq W$.  
 
  The quasicomponents of $O$ are quasicomponents of $Z \setminus H$.  Thus $a$ and $d$ are in the same quasicomponent of $O$ by Lemma \ref{L:quas}. Since $C \subseteq H \cup W$,  $O \subseteq Z \setminus C$.  By Lemma \ref{L:qinc}, $a$ and $d$ are in the same quasicomponent of $Z \setminus C$.  This proves that $C$ doesn't separate $A$.
  
All that is left to prove is (IV).  The case where $C \in \cA$ follows from the definition of our betweenness relation and Lemma \ref{L:quas}, so we are left with the case where $C\in \cB$.
 
 We first show (a) that is, if $C \in (A,B)$ then $C$ separates $A$ from $B$.  Now  $A \neq B$ and so $(A,C) \cap (C,B) \cap\cA= \emptyset$. Since $A, B \not \subseteq C$, then by Lemma  \ref{L:cutblob}, $(A,C) \neq \emptyset \neq (C,B)$. Using Lemma \ref{L:blob} we see that  $(A,C)\cap \cA \neq \emptyset \neq (C,B)\cap \cA$.  For each $D \in (A,C) \cap \cA$  use Lemma \ref{L:refine} to  choose a separation $(U_D, V_D)$ of $Z\setminus D$ with $A \setminus D \subseteq U_D$ and $[B\setminus D] \cup [C\setminus D] \subseteq V_D$.  
 Define the nested union $$U = \bigcup\limits_{D \in (A,C) \cap \cA} U_D$$ so $U$ is an open set  and $U \cap A \neq \emptyset$.  Let $V = Z \setminus [U \cup C]$.

 Clearly $V \cap U = \emptyset$.  We must show that $V \cap B \neq \emptyset$ and  that $V$ is open in $Z$.  
 
 Let $E \in (C,B) \cap \cA$ and use Lemma \ref{L:refine} to choose a partition $(O,W)$ of $\Z \setminus E$ with $[A\setminus E]\cup [C\setminus E] \subseteq O$ and $B \setminus E \subseteq W$.  By Lemma \ref{L:contain} for any $D \in (A,C) \cap \cA$, $W \subseteq V_D$. Thus $W \subseteq V$.   
 
 We now show that $V$ is open. 
 \begin{Notes}
 \item $\exists D \in (A,C)$ with $(D,C) = \emptyset$.\hfill\break
 By Lemma \ref{L:cutblob} $D\in \cA$ and $D \subseteq C$.  It follows from nesting that $U = U_D$, which implies that $V= V_D\setminus C$ which is open since $C$ is closed.
 \item $\forall D \in (A,C),\, (D,C) \neq \emptyset$\hfill\break
   Let $g \in V$.  The point $g$ is contained in a maximal inseparable set $G$ (either a cut or a blob) so $g \in G \in \cP$ with  $G\neq C$.   Since  $g \not \in C$ there is $F \in \cA$ which separates $g$ from $C$  implying $F \in (C,G)$.   
\begin{Subcase}
 \item  $C \in (A,G)$. \hfill\break
 Since $F \in(C,G) \subseteq (A,G)$
  Now by Lemma \ref{L:refine} there is a partition $(O,W)$ of $Z \setminus F$ with $G\setminus F \subseteq W$ and $[C\setminus F ]\cup [A\setminus F] \subseteq O$.
 Let $D \in (A,C) \cap \cA$.  Since $ F\not \in (A,D)$ then $D \setminus F \subseteq O$. Applying Lemma \ref{L:contain}  to $A, D, F,G$ we see that $U_D \cap W = \emptyset$.  It follows that $U\cap W= \emptyset$.  Since $C \cap W = \emptyset$ then $g \in W \subseteq V$.   Thus $V$ is open.  
 \item $C \not \in (A,G)$\hfill\break 
 By definition, there is $E \in (A,C)\cap(C,G) \cap \cA$.   Now by Lemma \ref{L:refine} there is a partition $(O,W)$ of $Z \setminus E$ with
 $[A\setminus E] \cup [G \setminus E] \subseteq O$ and $C \setminus E \subseteq W$. Since $g \not \in C$ we may choose $E$ with $g \not \in E$, so $g \in O$.  
 
   By our case and Lemma \ref{L:blob}, there is $D \in (E,C) \cap \cA$.  By Lemma \ref{L:contain} $O \subseteq U_D$. Thus $g \in O \subseteq U_D \subseteq U$ a contradiction.  
 \end{Subcase}
 \end{Notes}
 Thus $V$ is open and and so $C$ separates $A$ from $B$.  
 This completes the proof of (a).

 We now show (b), that if $C$ separates $A$ from $B$ then $D \in (A,B)$ for some $D \subseteq C$.  We may of course assume that $C \not \in (A,B)$. There is a separation $(U,V)$ of $Z\setminus C$ with $$A \cap U \neq \emptyset \neq B \cap V\,.$$  By (III)  $A \setminus C \subseteq U$ and $B \setminus C \subseteq V$.  
 Since $C$ is a blob and $A, B \not \subseteq C$, if $C \not \in (A,B)$ then by definition there is $D \in (A,C) \cap (B,C)\cap \cA$.  
 
 First consider the case where $D \subsetneq C$.  In this case $A \setminus C = A \setminus D \neq \emptyset$ and $B \setminus C = B \setminus D\neq\emptyset$.  
 Using Lemma \ref{L:refine}, there is a separation $(O,W)$ of $Z\setminus D$ with $A\setminus D, B\setminus D \subseteq O$ and $C\setminus D \subseteq W$.  Notice that $O \cap U$, $O \cap V$ and $W$ are disjoint  open subsets of $Z$.  They are nonempty since $A \setminus D \subseteq O \cap U$, $B \setminus D \subseteq O \cap V$ and $C\setminus D \subseteq W$.
 
 Since $C \cap O=\emptyset$,  then $O \subseteq U\cup V$, and so $O = (O \cap U)\cup (O\cap V)$.  Thus $Z \setminus D = (O\cap U) \sqcup (O\cap V) \sqcup W$.  Now we have that $(O \cap U, [O\cap V]\cup W)$ forms a separation of $Z \setminus D$ and so $D \in (A,B)$ by definition.  
 
 Lastly consider the case where $D \not \subseteq C$.  By Lemma \ref{L:cutblob} and Lemma \ref{L:blob} there is an $ E\in (D,C) \cap \cA$.
It follows using the natural linear oderings that $E \not \in [A,D]$ and $E \not \in [B,D]$.  Thus by pretree Axiom 4, $E \not \in (A,B)$.  By Corollary \ref{C:contain} $E \in (A,C)$, and so there is a separation $(U,V) $ of $Z\setminus E$ with $A \setminus E \subseteq U$ and $C\setminus E \subseteq V$.  Since $E\not\in (A,B)$, $B \setminus E \subseteq U$, and in fact $A \setminus E$ and $B \setminus E$ are in the same quasicomponent $Q$ of $Z \setminus E$.  Notice that $Q$ is also a quasicomponent of $U$.  However $C\setminus E \subseteq V$ and so $C \cap U = \emptyset$ and $U \subseteq Z \setminus C$.  By Lemma \ref{L:qinc}, $Q$ is contained in a single  quasicomponent of $Z \setminus C$.  Clearly $A \setminus C \subseteq A \setminus E$ and $B \setminus C \subseteq B \setminus E$.  Thus $A \setminus C, B \setminus C \subseteq Q$ which is contained in a single quasicomponent of $Z \setminus C$.  This contradicts $C$ separating $A$ from $B$, and (b) is proven.

\end{proof}

\end{document}
\]
\]
\end{align*}
\]
\]

\end{document}